\documentclass{amsart}
\usepackage{amssymb}
\usepackage{amsfonts}

\newtheorem{theorem}{Theorem}
\theoremstyle{plain}

\newtheorem{corollary}{Corollary}

\newtheorem{definition}{Definition}

\newtheorem{lemma}{Lemma}

\newtheorem{remark}{Remark}

\numberwithin{equation}{section}
\input{tcilatex}

\begin{document}
\title[INEQUALITIES VIA FRACTIONAL INTEGRALS]{HADAMARD TYPE INEQUALITIES FOR 
$m-$CONVEX AND $(\alpha ,m)-$CONVEX FUNCTIONS VIA FRACTIONAL INTEGRALS}
\author{M. Emin \"{O}zdemir$^{\blacklozenge }$}
\address{$^{\blacklozenge }$Ataturk University, K.K. Education Faculty,
Department of Mathematics, 25240, Erzurum, Turkey}
\email{emos@atauni.edu.tr}
\author{Merve Avc\i $^{\clubsuit ,\bigstar }$}
\address{$^{\clubsuit }$Ad\i yaman University, Faculty of Science and Arts,
Department of Mathematics, 02040, Ad\i yaman, Turkey}
\email{merveavci@ymail.com}
\author{Ahmet Ocak Akdemir$^{\spadesuit }$}
\address{$^{\spadesuit }$A\u{g}r\i\ \.{I}brahim \c{C}e\c{c}en University,
Faculty of Science and Arts, Department of Mathematics, 04100, A\u{g}r\i ,
Turkey}
\email{ahmetakdemir@agri.edu.tr}
\author{$^{\spadesuit }$Alper Ekinci}
\email{alperekinci@hotmail.com}
\thanks{$^{\bigstar }$Corresponding Author}

\begin{abstract}
In this paper, we established some new Hadamard-type integral inequalities
for functions whose derivatives of absolute values are $m-$convex and $%
(\alpha ,m)-$convex functions via Riemann-Liouville fractional integrals.
\end{abstract}

\maketitle

\section{Introduction}

Let $f:I\subset 
%TCIMACRO{\U{211d} }%
%BeginExpansion
\mathbb{R}
%EndExpansion
\rightarrow 
%TCIMACRO{\U{211d} }%
%BeginExpansion
\mathbb{R}
%EndExpansion
$ be a convex function defined on the interval $I$ of real numbers and $%
a,b\in I$ with $a<b$. The following inequality is well known in the
literature as the Hermite--Hadamard inequality:%
\begin{equation*}
f\left( \frac{a+b}{2}\right) \leq \frac{1}{b-a}\int_{a}^{b}f(x)dx\leq \frac{%
f(a)+f(b)}{2}.
\end{equation*}

In \cite{T}$,$ G.Toader defined the concept of m-convexity as the following:

\begin{definition}
\label{def 1.1} The function $f:[0,b]\rightarrow 
%TCIMACRO{\U{211d} }%
%BeginExpansion
\mathbb{R}
%EndExpansion
$ is said to be $m-$convex, where $m\in \lbrack 0,1],$ if for every $x,y\in
\lbrack 0,b]$ and $t\in \lbrack 0,1]$ we have:%
\begin{equation*}
f(tx+m(1-t)y)\leq tf(x)+m(1-t)f(y).
\end{equation*}%
Denote by $K_{m}(b)$ the set of the $m-$convex functions on $[0,b]$ for
which $f(0)\leq 0.$
\end{definition}

Several papers have been written on $m-$convex functions and we refer the
papers \cite{BOP}-\cite{DT}.

In \cite{DT}, the following inequality of Hermite-Hadamard type for $m-$%
convex functions holds:

\begin{theorem}
\label{teo 1.1} Let $f:[0,\infty )\rightarrow 
%TCIMACRO{\U{211d} }%
%BeginExpansion
\mathbb{R}
%EndExpansion
$ be a $m-$convex function with $m\in (0,1].$ If $0\leq a<b<\infty $ and $%
f\in L_{1}[a,b],$ then one has the inequality: 
\begin{equation}
\frac{1}{b-a}\int_{a}^{b}f(x)dx\leq \min \left\{ \frac{f(a)+mf(\frac{b}{m})}{%
2},\frac{f(b)+mf(\frac{a}{m})}{2}\right\} .  \label{1.1}
\end{equation}%
In \cite{SS}, S.S. Dragomir proved the following theorem.
\end{theorem}

\begin{theorem}
\label{teo 1.2} Let $f:[0,\infty )\rightarrow 
%TCIMACRO{\U{211d} }%
%BeginExpansion
\mathbb{R}
%EndExpansion
$ be a $m-$convex function with $m\in (0,1].$ If $f\in L_{1}[am,b]$ where $%
0\leq a<b,$ then one has the inequality:%
\begin{eqnarray}
&&\frac{1}{m+1}\left[ \frac{1}{mb-a}\int_{a}^{mb}f(x)dx+\frac{1}{b-ma}%
\int_{ma}^{b}f(x)dx\right]  \label{1.2} \\
&\leq &\frac{f(a)+f(b)}{2}.  \notag
\end{eqnarray}%
In \cite{MIH}, Mihe\c{s}an gave definition of $(\alpha ,m)-$convexity as
following;
\end{theorem}

\begin{definition}
The function $f:[0,b]\rightarrow 
%TCIMACRO{\U{211d} }%
%BeginExpansion
\mathbb{R}
%EndExpansion
,$ $b>0$ is said to be $(\alpha ,m)-$convex, where $(\alpha ,m)\in \lbrack
0,1]^{2},$ if we have%
\begin{equation*}
f(tx+m(1-t)y)\leq t^{\alpha }f(x)+m(1-t^{\alpha })f(y)
\end{equation*}%
for all $x,y\in \lbrack 0,b]$ and $t\in \lbrack 0,1].$
\end{definition}

Denote by $K_{m}^{\alpha }(b)$ the class of all $(\alpha ,m)-$convex
functions on $[0,b]$ for which $f(0)\leq 0.$ If we choose $(\alpha ,m)=(1,m)$%
, it can be easily seen that $(\alpha ,m)-$convexity reduces to $m-$%
convexity and for $(\alpha ,m)=(1,1),$ we have ordinary convex functions on $%
[0,b].$ For the recent results based on the above definition see the papers 
\cite{BOP}, \cite{BPR}, \cite{3}, \cite{SET}, \cite{10}, and \cite{5}.

We give some necessary definitions and mathematical preliminaries of
fractional calculus theory which are used throughout this paper.

\begin{definition}
\label{def 1.2} Let $f\in L_{1}[a,b].$ The Riemann-Liouville integrals $%
J_{a^{+}}^{\alpha }f$ and $J_{b^{-}}^{\alpha }f$ of order $\alpha >0$ with $%
a\geq 0$ are defined by%
\begin{equation*}
J_{a^{+}}^{\alpha }f(x)=\frac{1}{\Gamma (\alpha )}\int_{a}^{x}(x-t)^{\alpha
-1}f(t)dt,\text{ \ \ \ \ }x>a
\end{equation*}%
and 
\begin{equation*}
J_{b^{-}}^{\alpha }f(x)=\frac{1}{\Gamma (\alpha )}\int_{x}^{b}(t-x)^{\alpha
-1}f(t)dt,\text{ \ \ \ \ }x<b
\end{equation*}%
where $\Gamma (\alpha )=\int_{0}^{\infty }e^{-t}u^{\alpha -1}du,$ here is $%
J_{a^{+}}^{0}f(x)=J_{b^{-}}^{0}f(x)=f(x).$
\end{definition}

In the case of $\alpha =1,$ the fractional integral reduces to the classical
integral. Properties of this operator can be found in the references \cite%
{GM}-\cite{P}.

In \cite{SSYB}, Sarikaya \textit{et al.} proved a variant of the identity is
established by Dragomir and Agarwal in \cite[Lemma 2.1]{DA} for fractional
integrals as the following.

\begin{lemma}
\label{lem 1.1} Let $f:[a,b]\rightarrow 
%TCIMACRO{\U{211d} }%
%BeginExpansion
\mathbb{R}
%EndExpansion
$ be a differentiable mapping on $(a,b)$ with $a<b.$ If $f^{\prime }\in
L[a,b],$ then the following equality for fractional integrals holds:%
\begin{eqnarray*}
&&\frac{f(a)+f(b)}{2}-\frac{\Gamma (\alpha +1)}{2(b-a)^{\alpha }}\left[
J_{a^{+}}^{\alpha }f(b)+J_{b^{-}}^{\alpha }f(a)\right] \\
&=&\frac{b-a}{2}\int_{0}^{1}\left[ \left( 1-t\right) ^{\alpha }-t^{\alpha }%
\right] f^{\prime }(ta+(1-t)b)dt.
\end{eqnarray*}
\end{lemma}

The aim of this paper is to establish Hadamard type inequalities for $m-$%
convex and $(\alpha ,m)-$convex functions via Riemann-Liouville fractional
integrals.

\section{Inequalities for $m-$convex functions}

\begin{theorem}
\label{teo 2.1} Let $f:[a,b]\rightarrow 
%TCIMACRO{\U{211d} }%
%BeginExpansion
\mathbb{R}
%EndExpansion
$ be a positive function with $0\leq a<b$ and $f\in L_{1}[a,b].$ If $f$ is a 
$m-$convex function on $[a,b],$ then the following inequalities for
fractional integrals with $\alpha >0$ and $m\in (0,1]$ hold:%
\begin{eqnarray}
&&  \label{2.1} \\
\frac{\Gamma (\alpha )}{(b-a)^{\alpha }}J_{a^{+}}^{\alpha }f(b) &\leq &\frac{%
f(a)}{\alpha +1}+mf\left( \frac{b}{m}\right) \frac{\Gamma (\alpha )\Gamma (2)%
}{\Gamma (\alpha +2)},  \notag \\
\frac{\Gamma (\alpha )}{(b-a)^{\alpha }}J_{b^{-}}^{\alpha }f(a) &\leq &\frac{%
f(b)}{\alpha +1}+mf\left( \frac{a}{m}\right) \frac{\Gamma (\alpha )\Gamma (2)%
}{\Gamma (\alpha +2)}.  \notag
\end{eqnarray}%
where $\Gamma $ is Euler Gamma function.
\end{theorem}

\begin{proof}
Since $f$ is a $m-$convex function on $[a,b],$ we know that for any $t\in
\lbrack 0,1]$%
\begin{equation}
f(ta+(1-t)b)\leq tf(a)+m(1-t)f\left( \frac{b}{m}\right)  \label{2.2}
\end{equation}%
and 
\begin{equation}
f(tb+(1-t)a)\leq tf(b)+m(1-t)f\left( \frac{a}{m}\right) .  \label{2.3}
\end{equation}%
By multiplying both sides of (\ref{2.2}) by $t^{\alpha -1},$ then by
integrating the resulting inequality with respect to $t$ over $[0,1],$ we
obtain%
\begin{eqnarray*}
\int_{0}^{1}t^{\alpha -1}f(ta+(1-t)b)dt &\leq &\int_{0}^{1}t^{\alpha -1} 
\left[ tf(a)+m(1-t)f\left( \frac{b}{m}\right) \right] \\
&=&\frac{f(a)}{\alpha +1}+mf\left( \frac{b}{m}\right) \beta (\alpha ,2).
\end{eqnarray*}%
It is easy to see that $\int_{0}^{1}t^{\alpha -1}f(ta+(1-t)b)dt=\frac{\Gamma
(\alpha )}{(b-a)^{\alpha }}J_{a^{+}}^{\alpha }f(b)$ and we note that, the
Beta and the Gamma function (see \cite[pp 908-910]{GR}) 
\begin{equation*}
\beta (x,y)=\int_{0}^{1}t^{x-1}\left( 1-t\right) ^{y-1}dt,\text{ \ }x,y>0,%
\text{ \ \ \ \ }\Gamma (x)=\int_{0}^{\infty }e^{-t}t^{x-1}dt,\text{ \ }x>0
\end{equation*}%
are used to evaluate the integral%
\begin{equation*}
\int_{0}^{1}t^{\alpha -1}(1-t)dt,
\end{equation*}%
where 
\begin{equation*}
\beta (x,y)=\frac{\Gamma (x)\Gamma (y)}{\Gamma (x+y)},
\end{equation*}%
thus we can obtain that%
\begin{equation*}
\beta (\alpha ,2)=\frac{\Gamma (\alpha )\Gamma (2)}{\Gamma (\alpha +2)}
\end{equation*}%
which completes the proof.

For the proof of the second inequality in (\ref{2.1}) we multiply both sides
of (\ref{2.3}) by $t^{\alpha -1},$ then integrate the resulting inequality
with respect to $t$ over $[0,1].$
\end{proof}

\begin{remark}
\label{rem 2.1} If we choose $\alpha =1$ in Theorem \ref{teo 2.1}, then the
inequalities (\ref{2.1}) become the inequality in (1.1).
\end{remark}

\begin{theorem}
\label{teo 2.2} Let $f:I^{\circ }\subset \lbrack 0,\infty )\rightarrow 
%TCIMACRO{\U{211d} }%
%BeginExpansion
\mathbb{R}
%EndExpansion
$, be a differentiable function on $I^{\circ }$ such that $f^{\prime }\in
L[a,b]$ where $a,b\in I,a<b.$ If $\left\vert f^{\prime }\right\vert ^{q}$ is 
$m-$convex on $[a,b]$ for some fixed $m\in (0,1]$ and $q\geq 1,$ then the
following inequality for fractional integrals holds:%
\begin{eqnarray*}
&&\left\vert \frac{f(a)+f(b)}{2}-\frac{\Gamma (\alpha +1)}{2(b-a)^{\alpha }}%
\left[ J_{a^{+}}^{\alpha }f(b)+J_{b^{-}}^{\alpha }f(a)\right] \right\vert \\
&\leq &\frac{b-a}{2}2^{1-\frac{1}{q}}\left[ \frac{2^{\alpha }-1}{2^{\alpha
}(\alpha +1)}\right] \left[ \left\vert f^{\prime }(a)\right\vert
^{q}+m\left\vert f^{\prime }\left( \frac{b}{m}\right) \right\vert ^{q}\right]
^{\frac{1}{q}}.
\end{eqnarray*}
\end{theorem}

\begin{proof}
Suppose that $q=1.$ From Lemma \ref{lem 1.1} and by using the properties of
modulus, we have 
\begin{eqnarray*}
&&\left\vert \frac{f(a)+f(b)}{2}-\frac{\Gamma (\alpha +1)}{2(b-a)^{\alpha }}%
\left[ J_{a^{+}}^{\alpha }f(b)+J_{b^{-}}^{\alpha }f(a)\right] \right\vert \\
&\leq &\frac{b-a}{2}\int_{0}^{1}\left\vert \left( 1-t\right) ^{\alpha
}-t^{\alpha }\right\vert \left\vert f^{\prime }(ta+(1-t)b)\right\vert dt.
\end{eqnarray*}%
Since $\left\vert f^{\prime }\right\vert $ is $m-$convex on $[a,b],$ we have 
\begin{eqnarray*}
&&\left\vert \frac{f(a)+f(b)}{2}-\frac{\Gamma (\alpha +1)}{2(b-a)^{\alpha }}%
\left[ J_{a^{+}}^{\alpha }f(b)+J_{b^{-}}^{\alpha }f(a)\right] \right\vert \\
&\leq &\frac{b-a}{2}\int_{0}^{1}\left\vert \left( 1-t\right) ^{\alpha
}-t^{\alpha }\right\vert \left[ t\left\vert f^{\prime }(a)\right\vert
+m(1-t)\left\vert f^{\prime }\left( \frac{b}{m}\right) \right\vert \right] dt
\\
&=&\frac{b-a}{2}\left\{ \int_{0}^{\frac{1}{2}}\left[ \left( 1-t\right)
^{\alpha }-t^{\alpha }\right] \left[ t\left\vert f^{\prime }(a)\right\vert
+m(1-t)\left\vert f^{\prime }\left( \frac{b}{m}\right) \right\vert \right]
dt\right. \\
&&\left. +\int_{\frac{1}{2}}^{1}\left[ t^{\alpha }-\left( 1-t\right)
^{\alpha }\right] \left[ t\left\vert f^{\prime }(a)\right\vert
+m(1-t)\left\vert f^{\prime }\left( \frac{b}{m}\right) \right\vert \right]
dt\right\} \\
&=&\frac{b-a}{2}\left[ \frac{2^{\alpha }-1}{2^{\alpha }(\alpha +1)}\right] %
\left[ \left\vert f^{\prime }(a)\right\vert +m\left\vert f^{\prime }\left( 
\frac{b}{m}\right) \right\vert \right]
\end{eqnarray*}%
where we use the facts that%
\begin{equation*}
\int_{0}^{\frac{1}{2}}\left( 1-t\right) ^{\alpha }tdt=\int_{\frac{1}{2}%
}^{1}t^{^{\alpha }}\left( 1-t\right) dt=\frac{1}{\left( \alpha +1\right)
\left( \alpha +2\right) }-\frac{\alpha +3}{2^{\alpha +2}\left( \alpha
+1\right) \left( \alpha +2\right) },
\end{equation*}%
\begin{equation*}
\int_{0}^{\frac{1}{2}}t^{\alpha +1}dt=\int_{\frac{1}{2}}^{1}\left(
1-t\right) ^{\alpha +1}dt=\frac{1}{2^{\alpha +2}\left( \alpha +2\right) },
\end{equation*}%
\begin{equation*}
\int_{0}^{\frac{1}{2}}\left( 1-t\right) ^{\alpha +1}dt=\int_{\frac{1}{2}%
}^{1}t^{\alpha +1}dt=\frac{1}{\left( \alpha +2\right) }-\frac{1}{2^{\alpha
+2}\left( \alpha +2\right) }
\end{equation*}%
and%
\begin{equation*}
\int_{0}^{\frac{1}{2}}t^{^{\alpha }}\left( 1-t\right) dt=\int_{\frac{1}{2}%
}^{1}\left( 1-t\right) ^{\alpha }tdt=\frac{\alpha +3}{2^{\alpha +2}\left(
\alpha +1\right) \left( \alpha +2\right) }
\end{equation*}%
which completes the proof for this case. Suppose now that $q>1.$ From Lemma %
\ref{lem 1.1}, $m-$convexity of $\left\vert f^{\prime }\right\vert ^{q}$ and
using the well-known H\"{o}lder's inequality we have successively%
\begin{eqnarray*}
&&\left\vert \frac{f(a)+f(b)}{2}-\frac{\Gamma (\alpha +1)}{2(b-a)^{\alpha }}%
\left[ J_{a^{+}}^{\alpha }f(b)+J_{b^{-}}^{\alpha }f(a)\right] \right\vert \\
&\leq &\frac{b-a}{2}\int_{0}^{1}\left\vert \left( 1-t\right) ^{\alpha
}-t^{\alpha }\right\vert \left\vert f^{\prime }(ta+(1-t)b)\right\vert dt \\
&=&\frac{b-a}{2}\int_{0}^{1}\left\vert \left( 1-t\right) ^{\alpha
}-t^{\alpha }\right\vert ^{1-\frac{1}{q}}\left\vert \left( 1-t\right)
^{\alpha }-t^{\alpha }\right\vert ^{\frac{1}{q}}\left\vert f^{\prime
}(ta+(1-t)b)\right\vert dt \\
&\leq &\frac{b-a}{2}\left( \int_{0}^{1}\left\vert \left( 1-t\right) ^{\alpha
}-t^{\alpha }\right\vert dt\right) ^{\frac{q-1}{q}}\left(
\int_{0}^{1}\left\vert \left( 1-t\right) ^{\alpha }-t^{\alpha }\right\vert
\left\vert f^{\prime }(ta+(1-t)b)\right\vert ^{q}dt\right) ^{\frac{1}{q}} \\
&\leq &\frac{b-a}{2}2^{1-\frac{1}{q}}\left[ \frac{2^{\alpha }-1}{2^{\alpha
}(\alpha +1)}\right] \left[ \left\vert f^{\prime }(a)\right\vert
^{q}+m\left\vert f^{\prime }\left( \frac{b}{m}\right) \right\vert ^{q}\right]
^{\frac{1}{q}}
\end{eqnarray*}%
where we use the fact that%
\begin{eqnarray*}
\int_{0}^{1}\left\vert \left( 1-t\right) ^{\alpha }-t^{\alpha }\right\vert
dt &=&\int_{0}^{\frac{1}{2}}\left[ \left( 1-t\right) ^{\alpha }-t^{\alpha }%
\right] dt+\int_{\frac{1}{2}}^{1}\left[ t^{\alpha }-\left( 1-t\right)
^{\alpha }\right] dt \\
&=&\frac{2}{\alpha +1}\left( 1-\frac{1}{2^{\alpha }}\right) .
\end{eqnarray*}
\end{proof}

\begin{corollary}
\label{co 2.1} Under the assumptions of Theorem \ref{teo 2.2} with $\alpha
\in (0,1],$ the following inequality holds:%
\begin{eqnarray*}
&&\left\vert \frac{f(a)+f(b)}{2}-\frac{\Gamma (\alpha +1)}{2(b-a)^{\alpha }}%
\left[ J_{a^{+}}^{\alpha }f(b)+J_{b^{-}}^{\alpha }f(a)\right] \right\vert \\
&\leq &\frac{b-a}{2}\left( \frac{1}{\alpha p+1}\right) ^{\frac{1}{p}}\left( 
\frac{\left\vert f^{\prime }(a)\right\vert ^{q}+m\left\vert f^{\prime
}\left( \frac{b}{m}\right) \right\vert ^{q}}{2}\right) ^{\frac{1}{q}}.
\end{eqnarray*}
\end{corollary}

\begin{proof}
It is similar the proof of Theorem \ref{teo 2.2}. In addition, we used the
following inequality 
\begin{equation*}
\left\vert t_{1}^{\alpha }-t_{2}^{\alpha }\right\vert \leq \left\vert
t_{1}-t_{2}\right\vert ^{\alpha },
\end{equation*}%
where $\alpha \in (0,1]$ and $t_{1},t_{2}\in \lbrack 0,1].$
\end{proof}

\begin{theorem}
\label{teo 2.3} Let $f:[0,\infty )\rightarrow 
%TCIMACRO{\U{211d} }%
%BeginExpansion
\mathbb{R}
%EndExpansion
$ be a $m-$convex function with $m\in (0,1].$ If $f\in L_{1}[am,b],$ where $%
0\leq a<b,$ then the following inequality for fractional integrals with $%
\alpha >0$ holds:%
\begin{eqnarray}
&&  \label{2.4} \\
&&\frac{\Gamma (\alpha )}{m+1}\left\{ \frac{1}{\left( mb-a\right) ^{\alpha }}%
\left[ J_{a^{+}}^{\alpha }f(mb)+J_{mb^{-}}^{\alpha }f(mb)\right] +\frac{1}{%
\left( b-ma\right) ^{\alpha }}\left[ J_{ma^{+}}^{\alpha
}f(mb)+J_{b^{-}}^{\alpha }f(mb)\right] \right\}  \notag \\
&\leq &\frac{f(a)+f(b)}{\alpha }.  \notag
\end{eqnarray}
\end{theorem}

\begin{proof}
By the $m-$convexity of $f$ we can write 
\begin{equation*}
f(ta+m(1-t)b)\leq tf(a)+m(1-t)f(b),
\end{equation*}%
\begin{equation*}
f((1-t)a+mtb)\leq (1-t)f(a)+mtf(b),
\end{equation*}%
\begin{equation*}
f(tb+m(1-t)a)\leq tf(b)+m(1-t)f(a)
\end{equation*}%
and%
\begin{equation*}
f((1-t)b+mta)\leq (1-t)f(b)+mtf(a)
\end{equation*}%
for all $t\in \lbrack 0,1]$ and $0\leq a<b.$

If we add the above inequalities we get%
\begin{eqnarray}
&&  \label{2.5} \\
&&f(ta+m(1-t)b)+f((1-t)a+mtb)+f(tb+m(1-t)a)+f((1-t)b+mta)  \notag \\
&\leq &\left( m+1\right) \left( f(a)+f(b)\right) .  \notag
\end{eqnarray}%
By multiplying both sides of (\ref{2.5}) with $t^{\alpha -1},$ then
integrating the resulting inequality with respect to $t$ over $[0,1],$ we
obtain%
\begin{eqnarray*}
&&\int_{0}^{1}t^{\alpha -1}\left[
f(ta+m(1-t)b)+f((1-t)a+mtb)+f(tb+m(1-t)a)+f((1-t)b+mta)\right] dt \\
&\leq &\left( m+1\right) \left( f(a)+f(b)\right) \int_{0}^{1}t^{\alpha -1}dt.
\end{eqnarray*}%
It is easy to see that%
\begin{equation*}
\int_{0}^{1}t^{\alpha -1}f(ta+m(1-t)b)dt=\frac{\Gamma (\alpha )}{\left(
mb-a\right) ^{\alpha }}J_{a^{+}}^{\alpha }f(mb),
\end{equation*}%
\begin{equation*}
\int_{0}^{1}t^{\alpha -1}f((1-t)a+mtb)dt=\frac{\Gamma (\alpha )}{\left(
mb-a\right) ^{\alpha }}J_{mb^{-}}^{\alpha }f(mb),
\end{equation*}%
\begin{equation*}
\int_{0}^{1}t^{\alpha -1}f(tb+m(1-t)a)dt=\frac{\Gamma (\alpha )}{\left(
b-ma\right) ^{\alpha }}J_{b^{-}}^{\alpha }f(mb)
\end{equation*}%
and%
\begin{equation*}
\int_{0}^{1}t^{\alpha -1}f((1-t)b+mta)dt=\frac{\Gamma (\alpha )}{\left(
b-ma\right) ^{\alpha }}J_{ma^{+}}^{\alpha }f(mb).
\end{equation*}%
So the proof is completed.
\end{proof}

\begin{remark}
\label{rem 2.2} If we choose $\alpha =1$ in Theorem \ref{teo 2.3}, then the
inequality (\ref{2.4}) becomes the inequality in (\ref{1.2}).
\end{remark}

\section{Inequalities for $\left( \protect\alpha ,m\right) -$convex functions%
}

\begin{theorem}
\label{a0}Let $f:[a,b]\rightarrow 
%TCIMACRO{\U{211d} }%
%BeginExpansion
\mathbb{R}
%EndExpansion
$ be a positive function with $0\leq a<b$ and $f\in L_{1}[a,b].$ If $f$ is
an $\left( \alpha _{1},m\right) -$convex function on $[a,b],$ then the
following inequalities hold for fractional integrals with $\alpha >0$ and $%
\left( \alpha _{1},m\right) \in (0,1]^{2}$:%
\begin{eqnarray}
&&  \label{a} \\
\frac{\Gamma (\alpha )}{(b-a)^{\alpha }}J_{a^{+}}^{\alpha }f(b) &\leq &\frac{%
1}{\alpha +\alpha _{1}}f(a)+\frac{m\alpha _{1}}{\alpha \left( \alpha +\alpha
_{1}\right) }f\left( \frac{b}{m}\right) ,  \notag \\
\frac{\Gamma (\alpha )}{(b-a)^{\alpha }}J_{b^{-}}^{\alpha }f(a) &\leq &\frac{%
1}{\alpha +\alpha _{1}}f(b)+\frac{m\alpha _{1}}{\alpha \left( \alpha +\alpha
_{1}\right) }f\left( \frac{a}{m}\right) .  \notag
\end{eqnarray}
\end{theorem}

\begin{proof}
Since $f$ is $\left( \alpha _{1},m\right) -$convex function on $[a,b],$ we
know that for any $t\in \lbrack 0,1]$%
\begin{equation}
f(ta+(1-t)b)\leq t^{\alpha _{1}}f(a)+m(1-t^{\alpha _{1}})f\left( \frac{b}{m}%
\right)  \label{a1}
\end{equation}%
and 
\begin{equation}
f(tb+(1-t)a)\leq t^{\alpha _{1}}f(b)+m(1-t^{\alpha _{1}})f\left( \frac{a}{m}%
\right) .  \label{a2}
\end{equation}%
By multiplying both sides of (\ref{a1}) by $t^{\alpha -1},$ then by
integrating the resulting inequality with respect to $t$ over $[0,1],$ we get%
\begin{eqnarray*}
\int_{0}^{1}t^{\alpha -1}f(ta+(1-t)b)dt &\leq &\int_{0}^{1}t^{\alpha -1} 
\left[ t^{\alpha _{1}}f(a)+m(1-t^{\alpha _{1}})f\left( \frac{b}{m}\right) %
\right] \\
&=&\frac{1}{\alpha +\alpha _{1}}f(a)+\frac{m\alpha _{1}}{\alpha \left(
\alpha +\alpha _{1}\right) }f\left( \frac{b}{m}\right) .
\end{eqnarray*}%
It is easy to see that $\int_{0}^{1}t^{\alpha -1}f(ta+(1-t)b)dt=\frac{\Gamma
(\alpha )}{(b-a)^{\alpha }}J_{a^{+}}^{\alpha }f(b)$, by using this fact the
proof of the first inequality is completed.

Similarly, by multiplying both sides of (\ref{a2}) by $t^{\alpha -1},$ then
by integration, the proof of the second inequality is completed.
\end{proof}

\begin{corollary}
If we choose $\alpha =\alpha _{1}$ with $\alpha ,\alpha _{1}\in (0,1]$ in
Theorem \ref{a0}$,$ we have the following inequalities;%
\begin{eqnarray}
\frac{\Gamma (\alpha )}{(b-a)^{\alpha }}J_{a^{+}}^{\alpha }f(b) &\leq &\frac{%
1}{2\alpha }\left[ f(a)+mf\left( \frac{b}{m}\right) \right] ,  \notag \\
\frac{\Gamma (\alpha )}{(b-a)^{\alpha }}J_{b^{-}}^{\alpha }f(a) &\leq &\frac{%
1}{2\alpha }\left[ f(b)+mf\left( \frac{a}{m}\right) \right] .  \notag
\end{eqnarray}
\end{corollary}

\begin{remark}
If we choose $\alpha =\alpha _{1}=1$ in Theorem \ref{a0}$,$ then the
inequalities reduces to the inequality (\ref{1.1}).
\end{remark}

\begin{theorem}
\label{a3}Let $f:I^{\circ }\subset \lbrack 0,\infty )\rightarrow 
%TCIMACRO{\U{211d} }%
%BeginExpansion
\mathbb{R}
%EndExpansion
$, be a differentiable function on $I^{\circ }$ such that $f^{\prime }\in
L[a,b]$ where $a,b\in I,$ $a<b.$ If $\left\vert f^{\prime }\right\vert ^{q}$
is an $\left( \alpha _{1},m\right) -$convex function and $\left\vert
f^{\prime }\right\vert $ is decreasing on $[a,b],$ then the following
inequalities hold for fractional integrals with $\alpha >0$, $\left( \alpha
_{1},m\right) \in (0,1]^{2}$ ; 
\begin{eqnarray*}
&&\left\vert \frac{f(a)+f(b)}{2}-\frac{\Gamma (\alpha +1)}{2(b-a)^{\alpha }}%
\left[ J_{a^{+}}^{\alpha }f(b)+J_{b^{-}}^{\alpha }f(a)\right] \right\vert \\
&\leq &\frac{b-a}{2}\left[ \frac{2^{\alpha }-1}{2^{\alpha -1}\left( \alpha
+1\right) }\right] ^{\frac{q-1}{q}}\left[ \left( \frac{2^{\alpha +\alpha
_{1}}-1}{2^{\alpha +\alpha _{1}}\left( \alpha +\alpha _{1}+1\right) }\right)
\left( \left\vert f^{\prime }(a)\right\vert ^{q}-m\left\vert f^{\prime
}\left( \frac{b}{m}\right) \right\vert ^{q}\right) \right. \\
&&\left. +\frac{m}{\alpha +1}\left\vert f^{\prime }\left( \frac{b}{m}\right)
\right\vert ^{q}\left( 1-\frac{1}{2^{\alpha }}\right) \right] ^{\frac{1}{q}}.
\end{eqnarray*}
\end{theorem}

\begin{proof}
Suppose that $q=1.$ From Lemma \ref{lem 1.1} and by using the properties of
modulus, we have 
\begin{eqnarray*}
&&\left\vert \frac{f(a)+f(b)}{2}-\frac{\Gamma (\alpha +1)}{2(b-a)^{\alpha }}%
\left[ J_{a^{+}}^{\alpha }f(b)+J_{b^{-}}^{\alpha }f(a)\right] \right\vert \\
&\leq &\frac{b-a}{2}\int_{0}^{1}\left\vert \left( 1-t\right) ^{\alpha
}-t^{\alpha }\right\vert \left\vert f^{\prime }(ta+(1-t)b)\right\vert dt.
\end{eqnarray*}%
Since $\left\vert f^{\prime }\right\vert $ is $\left( \alpha _{1},m\right) -$%
convex on $[a,b],$ we have 
\begin{eqnarray*}
&&\left\vert \frac{f(a)+f(b)}{2}-\frac{\Gamma (\alpha +1)}{2(b-a)^{\alpha }}%
\left[ J_{a^{+}}^{\alpha }f(b)+J_{b^{-}}^{\alpha }f(a)\right] \right\vert \\
&\leq &\frac{b-a}{2}\int_{0}^{1}\left\vert \left( 1-t\right) ^{\alpha
}-t^{\alpha }\right\vert \left[ t^{\alpha }\left\vert f^{\prime
}(a)\right\vert +m(1-t^{\alpha })\left\vert f^{\prime }\left( \frac{b}{m}%
\right) \right\vert \right] dt \\
&=&\frac{b-a}{2}\left\{ \int_{0}^{\frac{1}{2}}\left[ \left( 1-t\right)
^{\alpha }-t^{\alpha }\right] \left[ t^{\alpha }\left\vert f^{\prime
}(a)\right\vert +m(1-t^{\alpha })\left\vert f^{\prime }\left( \frac{b}{m}%
\right) \right\vert \right] dt\right. \\
&&\left. +\int_{\frac{1}{2}}^{1}\left[ t^{\alpha }-\left( 1-t\right)
^{\alpha }\right] \left[ t^{\alpha }\left\vert f^{\prime }(a)\right\vert
+m(1-t^{\alpha })\left\vert f^{\prime }\left( \frac{b}{m}\right) \right\vert %
\right] dt\right\} .
\end{eqnarray*}%
By computing the above integrals, we obtain%
\begin{eqnarray*}
&&\left\vert \frac{f(a)+f(b)}{2}-\frac{\Gamma (\alpha +1)}{2(b-a)^{\alpha }}%
\left[ J_{a^{+}}^{\alpha }f(b)+J_{b^{-}}^{\alpha }f(a)\right] \right\vert \\
&\leq &\frac{b-a}{2}\left[ \left( \frac{2^{\alpha +\alpha _{1}}-1}{2^{\alpha
+\alpha _{1}}\left( \alpha +\alpha _{1}+1\right) }\right) \left( \left\vert
f^{\prime }(a)\right\vert -m\left\vert f^{\prime }\left( \frac{b}{m}\right)
\right\vert \right) \right. \\
&&\left. +\frac{m}{\alpha +1}\left\vert f^{\prime }\left( \frac{b}{m}\right)
\right\vert \left( 1-\frac{1}{2^{\alpha }}\right) \right]
\end{eqnarray*}%
where we used the fact that%
\begin{equation*}
\int_{0}^{\frac{1}{2}}t^{\alpha _{1}}(1-t)^{\alpha }dt=\int_{\frac{1}{2}%
}^{1}t^{\alpha _{1}}(1-t)^{\alpha }dt=\beta (\frac{1}{2};\alpha
_{1}+1,\alpha +1).
\end{equation*}%
This completes the proof of this case. Suppose now that $q>1.$ Again, from
Lemma \ref{lem 1.1} and by applying well-known H\"{o}lder's inequality, we
have 
\begin{eqnarray*}
&&\left\vert \frac{f(a)+f(b)}{2}-\frac{\Gamma (\alpha +1)}{2(b-a)^{\alpha }}%
\left[ J_{a^{+}}^{\alpha }f(b)+J_{b^{-}}^{\alpha }f(a)\right] \right\vert \\
&\leq &\frac{b-a}{2}\left( \int_{0}^{1}\left\vert \left( 1-t\right) ^{\alpha
}-t^{\alpha }\right\vert dt\right) ^{\frac{q-1}{q}}\left(
\int_{0}^{1}\left\vert \left( 1-t\right) ^{\alpha }-t^{\alpha }\right\vert
\left\vert f^{\prime }(ta+(1-t)b)\right\vert ^{q}dt\right) ^{\frac{1}{q}}
\end{eqnarray*}%
By computing the above integrals and by using $\left( \alpha _{1},m\right) -$%
convexity of $\left\vert f^{\prime }\right\vert ^{q},$ we deduce 
\begin{eqnarray*}
&&\left\vert \frac{f(a)+f(b)}{2}-\frac{\Gamma (\alpha +1)}{2(b-a)^{\alpha }}%
\left[ J_{a^{+}}^{\alpha }f(b)+J_{b^{-}}^{\alpha }f(a)\right] \right\vert \\
&\leq &\frac{b-a}{2}\left[ \frac{2^{\alpha }-1}{2^{\alpha -1}\left( \alpha
+1\right) }\right] ^{\frac{q-1}{q}}\left[ \left( \frac{2^{\alpha +\alpha
_{1}}-1}{2^{\alpha +\alpha _{1}}\left( \alpha +\alpha _{1}+1\right) }\right)
\left( \left\vert f^{\prime }(a)\right\vert ^{q}-m\left\vert f^{\prime
}\left( \frac{b}{m}\right) \right\vert ^{q}\right) \right. \\
&&\left. +\frac{m}{\alpha +1}\left\vert f^{\prime }\left( \frac{b}{m}\right)
\right\vert ^{q}\left( 1-\frac{1}{2^{\alpha }}\right) \right] ^{\frac{1}{q}}.
\end{eqnarray*}%
Which completes the proof.
\end{proof}

\begin{corollary}
If we choose $\alpha =\alpha _{1}$ with $\alpha ,\alpha _{1}\in (0,1]$ in
Theorem \ref{a3}$,$ we have the inequality;%
\begin{eqnarray*}
&&\left\vert \frac{f(a)+f(b)}{2}-\frac{\Gamma (\alpha +1)}{2(b-a)^{\alpha }}%
\left[ J_{a^{+}}^{\alpha }f(b)+J_{b^{-}}^{\alpha }f(a)\right] \right\vert \\
&\leq &\frac{b-a}{2}\left[ \frac{2^{\alpha }-1}{2^{\alpha -1}\left( \alpha
+1\right) }\right] ^{\frac{q-1}{q}}\left[ \left( \frac{2^{2\alpha }-1}{%
2^{2\alpha }\left( 2\alpha +1\right) }\right) \left( \left\vert f^{\prime
}(a)\right\vert ^{q}-m\left\vert f^{\prime }\left( \frac{b}{m}\right)
\right\vert ^{q}\right) \right. \\
&&\left. +\frac{m}{\alpha +1}\left\vert f^{\prime }\left( \frac{b}{m}\right)
\right\vert ^{q}\left( 1-\frac{1}{2^{\alpha }}\right) \right] ^{\frac{1}{q}}.
\end{eqnarray*}
\end{corollary}

\begin{remark}
Identical result of Theorem \ref{teo 2.3}, can be stated for $\left( \alpha
,m\right) -$convex functions, but we omit the details.
\end{remark}

\end{document}